\pdfoutput=1
\documentclass[a4paper]{amsart} 
\usepackage{scrhack}
\usepackage[T1]{fontenc} 
\usepackage[utf8]{inputenc}
\usepackage[english]{babel}
\usepackage[babel]{csquotes}
\usepackage{lmodern} 
\usepackage[stretch=10]{microtype}
\usepackage{amssymb}
\usepackage{mathrsfs}
\usepackage{amsthm}
\usepackage{mathtools}
\usepackage{graphicx}
\usepackage{tikz}
\usepackage{tikz-cd}
\usepackage{fullpage}
\usepackage{fixltx2e} 
\usepackage[pdftex,
            pdfauthor={Jan Hesse},
            pdftitle={FROBENIUS ALGEBRAS AND HOMOTOPY FIXED POINTS OF GROUP
ACTIONS ON BICATEGORIES},
             breaklinks=true]{hyperref}

\newcommand{\Z}{\mathbb{Z}}

\newcommand{\K}{\mathbb{K}}

\newcommand{\sK}{\mathscr{K}}

\newcommand{\cC}{\mathcal{C}}
\newcommand{\cD}{\mathcal{D}}

\newcommand{\cM}{\mathcal{M}}

\newcommand{\B}[1]{B #1}
\newcommand{\ubar}[1]{\underline{#1}}
\newcommand{\core}[1]{\sK ({#1})}

\newcommand{\ori}{\textrm{or}}

\newcommand{\id}{\mathrm{id}}

\newcommand{\tr}{\mathrm{tr}}
\newcommand{\op}{\mathrm{op}}

\newcommand{\ot}{\otimes}
\newcommand{\eps}{\varepsilon}

\newcommand{\fd}{\mathrm{fd}}

\DeclareMathOperator{\Hom}{Hom}

\DeclareMathOperator{\Ob}{Ob}

\DeclareMathOperator{\Aut}{Aut}

\DeclareMathOperator{\Set}{Set}
\DeclareMathOperator{\Cob}{Cob}
\DeclareMathOperator{\Fun}{Fun}

\DeclareMathOperator{\Rep}{Rep}
\DeclareMathOperator{\Frob}{Frob}
\DeclareMathOperator{\CY}{CY}
\DeclareMathOperator{\Vect}{Vect}
\DeclareMathOperator{\Cat}{Cat}

\DeclareMathOperator{\ev}{ev}

\DeclareMathOperator{\End}{End}
\DeclareMathOperator{\Alg}{Alg}

\DeclareMathOperator{\Nat}{Nat}

\DeclareMathOperator{\Bicat}{Bicat}

\DeclareMathOperator{\Top}{Top}

\DeclareMathOperator{\BiGrp}{BiGrp}

\theoremstyle{definition}
\newtheorem{newdef}{Definition}[section]
\newtheorem{ex}[newdef]{Example}
\newtheorem{remark}[newdef]{Remark}

\theoremstyle{plain} 
\newtheorem{theorem}[newdef]{Theorem}
\newtheorem{lemma}[newdef]{Lemma}
\newtheorem{prop}[newdef]{Proposition}
\newtheorem{cor}[newdef]{Corollary}

\numberwithin{equation}{section}

\title{Frobenius algebras and homotopy fixed points of group actions on
  bicategories}
  \author{Jan Hesse}
  \author{Christoph Schweigert}
  \address{Fachbereich Mathematik, Universit\"at Hamburg, Bereich Algebra und Zahlentheorie, Bundesstraße 55, D – 20 146 Hamburg}
  \author{Alessandro Valentino}
  \address{Max Planck Institut f\"ur Mathematik, Vivatsgasse 7, 53111, Bonn}
  \date{}
 
\begin{document}
  \begin{abstract}
    We explicitly show that symmetric Frobenius structures on a
    finite-di\-men\-sion\-al, semi-simple algebra stand in bijection
    to homotopy fixed points of the trivial $SO(2)$-action on the
    bicategory of finite-dimensional, semi-simple algebras, bimodules
    and intertwiners.  The results are motivated by the 2-dimensional
    Cobordism Hypothesis for oriented manifolds, and can hence be
    interpreted in the realm of Topological Quantum Field Theory.
  \end{abstract}
  \begin{flushright}
    {ZMP-HH/16-13 \\ Hamburger Beitr\"age zur Mathematik Nr. 597}\\[1cm]
  \end{flushright}
  \maketitle
  \section{Introduction}
  While fixed points of a group action on a set form an ordinary
  subset, homotopy fixed points of a group action on a category as
  considered in \cite{kir02, egno-tensor-book} provide additional
  structure.

  In this paper, we take one more step on the categorical ladder by
  considering a topological group $G$ as a 3-group via its fundamental
  2-groupoid. We provide a detailed definition of an action of this
  3-group on an arbitrary bicategory $\cC$, and construct the
  bicategory of homotopy fixed points $\cC^G$ as a suitable limit of
  the action. Contrarily from the case of ordinary fixed points of
  group actions on sets, the bicategory of homotopy fixed points
  $\cC^G$ is strictly \enquote{larger} than the bicategory
  $\cC$. Hence, the usual fixed-point condition is promoted from a
  property to a structure.

  Our paper is motivated by the 2-dimensional Cobordism Hypothesis for
  oriented manifolds: according to \cite{Lurie09}, 2-dimensional
  oriented fully-extended topological quantum field theories are
  classified by homotopy fixed points of an $SO(2)$-action on the core
  of fully-dualizable objects of the symmetric monoidal target
  bicategory.  In case the target bicategory of a 2-dimensional
  oriented topological field theory is given by $\Alg_2$, the
  bicategory of algebras, bimodules and intertwiners, it is claimed in
  \cite[Example 2.13]{fhlt} that the additional structure of a
  homotopy fixed point should be given by the structure of a symmetric
  Frobenius algebra.

  As argued in \cite{Lurie09}, the $SO(2)$-action on $\Alg_2$ should
  come from rotating the 2-framings in the framed cobordism
  category. By \cite[Proposition 3.2.8]{davi11}, the induced
  action on the core of fully-dualizable objects of $\Alg_2$ is
  actually trivializable. Hence, instead of considering the action
  coming from the framing, we may equivalently study the
  \emph{trivial} $SO(2)$-action on $\Alg_2^\fd$.

  Our main result, namely Theorem \ref{thm:so2-fixed-point-bicat},
  computes the bicategory of homotopy fixed points $\cC^{SO(2)}$ of
  the trivial $SO(2)$-action on an arbitrary bicategory $\cC$. It
  follows then as a corollary that the bicategory
  $(\core{\Alg_2^\fd})^{SO(2)}$ consisting of homotopy fixed points of
  the trivial $SO(2)$-action on the core of fully-dualizable objects
  of $\Alg_2$ is equivalent to the bicategory $\Frob$ of semisimple
  symmetric Frobenius algebras, compatible Morita contexts, and
  intertwiners. This bicategory, or rather bigroupoid, classifies
  2-dimensional oriented fully-extended topological quantum field
  theories, as shown in \cite{schommerpries-classification}. Thus,
  unlike fixed points of the trivial action on a set, homotopy
  fixed-points of the trivial $SO(2)$-action on $\Alg_2$ are actually
  interesting, and come equipped with the additional structure of a
  symmetric Frobenius algebra.

  If $\Vect_2$ is the bicategory of linear abelian categories, linear
  functors and natural transformations, we show in corollary
  \ref{cor:cy-fixed-points} that the bicategory
  $(\core{\Vect_2^\fd})^{SO(2)}$ given by homotopy fixed points of the
  trivial $SO(2)$-action on the core of the fully dualizable objects
  of $\Vect_2$ is equivalent to the bicategory of Calabi-Yau
  categories, which we introduce in Definition
  \ref{def:calabi-yau-cat}.

  The two results above are actually intimately related to each other
  via natural considerations from representation theory. Indeed, by
  assigning to a finite-dimensional, semi-simple algebra its category
  of finitely-generated modules, we obtain a functor $\Rep: \core{
    \Alg_2^\fd} \to \core{ \Vect_2^\fd}$. This 2-functor turns out to
  be $SO(2)$-equivariant, and thus induces a morphism on homotopy
  fixed point bicategories, which is moreover an equivalence. More
  precisely, one can show that a symmetric Frobenius algebra is sent
  by the induced functor to its category of representations equipped
  with the Calabi-Yau structure given by the composite of the
  Frobenius form and the Hattori-Stallings trace. These results have
  appeared in \cite{hesse16}.

  The present paper is organized as follows: we recall the concept of
  Morita contexts between symmetric Frobenius algebras in section
  \ref{sec:frob-alg-morita-context}. Although most of the material has
  already appeared in \cite{schommerpries-classification}, we give
  full definitions to mainly fix the notation. We give a very explicit
  description of compatible Morita contexts between finite-dimensional
  semi-simple Frobenius algebras not present in
  \cite{schommerpries-classification}, which will be needed to relate
  the bicategory of symmetric Frobenius algebras and compatible Morita
  contexts to the bicategory of homotopy fixed points of the trivial
  $SO(2)$-action. The expert reader might wish to at least take notice
  of the notion of a compatible Morita context between symmetric
  Frobenius algebras in definition \ref{def:morita-compatible} and the
  resulting bicategory $\Frob$ in definition
  \ref{def:frob-bicategory}.

  In section \ref{sec:group-actions-bicat}, we recall the notion of a
  group action on a category and of its homotopy fixed points, which
  has been named \enquote{equivariantization} in \cite[Chapter
  2.7]{egno-tensor-book}. By categorifying this notion, we arrive at
  the definition of a group action on a bicategory and its homotopy
  fixed points. This definition is formulated in the language of
  tricategories. Since we prefer to work with bicategories, we
  explicitly spell out the definition in Remark
  \ref{rem:unpacking-fixed-point}.

  In section \ref{sec:results}, we compute the bicategory of homotopy
  fixed points of the trivial $SO(2)$-action on an arbitrary
  bicategory. Corollaries \ref{cor:frob-fixed-points} and
  \ref{cor:cy-fixed-points} then show equivalences of bicategories
  \begin{equation}
    \begin{aligned}
      (\core{\Alg_2^\fd})^{SO(2)} & \cong \Frob \\
      (\core{\Vect_2^\fd})^{SO(2)} & \cong \CY
    \end{aligned}
  \end{equation}
  where $\CY$ is the bicategory of Calabi-Yau categories. We note that
  the bicategory $\Frob$ has been proven to be equivalent \cite[Proposition
  3.3.2]{davi11} to a certain bicategory of 2-functors. We clarify the
  relationship between this functor bicategory and the bicategory of
  homotopy fixed points $(\core{\Alg_2^\fd})^{SO(2)}$ in Remark \ref{rem:davidovich}.

  Throughout the paper, we use the following conventions: all algebras
  considered will be over an algebraically closed field $\K$. All
  Frobenius algebras appearing will be symmetric.

\section*{Acknowledgments}
The authors would like to thank Ehud Meir for inspiring discussions and
Louis-Hadrien Robert for providing a proof of Lemma
\ref{lem:most-general-bimodule}. JH is supported by the RTG 1670
\enquote{Mathematics inspired by String Theory and Quantum Field
  Theory}. CS is partially supported by the Collaborative Research
Centre 676 \enquote{Particles, Strings and the Early Universe - the
  Structure of Matter and Space-Time} and by the RTG 1670
\enquote{Mathematics inspired by String Theory and Quantum Field
  Theory}. AV is supported by the \enquote{Max Planck Institut f\"ur Mathematik}.

\section{Frobenius algebras and Morita contexts}
In this section we will recall some basic notions regarding Morita
contexts, mostly with the aim of setting up notations. We will mainly
follow \cite{schommerpries-classification}, though we point the reader
to Remark \ref{rem:compatibility} for a slight difference in the
statement of the compatibility condition between Morita context and
Frobenius forms.
\label{sec:frob-alg-morita-context}
\begin{newdef}
  \label{def:morita-context}
  Let $A$ and $B$ be two algebras. A Morita context $\cM$ consists of
  a quadruple $\cM:=( {_{B}M_{A}}, {_{A}N_{B}},\eps,\eta)$, where
  $_BM_A$ is a $(B,A)$-bimodule, $_AN_B$ is an $(A,B)$-bimodule, and
  \begin{equation}
    \begin{aligned}
      \eps&: { _{A}N \otimes_{B}{ M_{A}}} \to {_{A}A_{A}} \\
      \eta&: {_BB_B} \to{ _{B}M \otimes_{A} {N_B}}
    \end{aligned}
  \end{equation}
  are isomorphisms of bimodules, so that the two diagrams
  \begin{equation}
    \label{eq:def-morita-context-1}
    \begin{tikzcd}[column sep=large]
      {_BM} \ot_A {N_B} \ot_B {M_A} \rar{\id_M \ot \eps}  & {_BM} \otimes_A {A_A} \dar \\
      {_BB} \ot_B {M_A} \uar{\eta \otimes \id_M} \rar & {_BM_A}
    \end{tikzcd}
  \end{equation}
  \begin{equation}
    \label{eq:def-morita-context-2}
    \begin{tikzcd}[column sep=large]
      {_AN} \ot_B {M} \ot_A {N_B} \dar{\eps \otimes \id_N} & {_AN} \otimes_B {B_B} \dar \lar{\id_N \ot \eta} \\
      {_AA} \ot_A {N_B} \rar & {_AN_B}
    \end{tikzcd}
  \end{equation}
  commute.
\end{newdef}
Note that Morita contexts are the adjoint 1-equivalences in the
bicategory $\Alg_2$ of algebras, bimodules and intertwiners. These
form a category, where the morphisms are given by the following:
\begin{newdef}
  \label{def:morphism-of-morita-context}
  Let $\cM:=( {_{B}M_{A}}, {_{A}N_{B}},\eps,\eta)$ and $\cM':=(
  {_{B}{M'}_{A}}, {_{A}{N'}_{B}},\eps',\eta')$ be two Morita contexts
  between two algebras $A$ and $B$. A morphism of Morita contexts
  consists of a morphism of $(B,A)$-bimodules $f:M \to M'$ and a
  morphism of $(A,B)$-bimodules $g:N \to N'$, so that the two diagrams
  \begin{equation}
    \begin{tikzcd}[row sep=large, column sep=large]
      _BM \ot_A N_B \rar{f \ot g} & {_BM' \ot_A N'_B} \\
      B \uar{\eta} \urar[swap]{\eta'}& {}
    \end{tikzcd}
    \qquad
    \begin{tikzcd}[row sep=large, column sep=large]
      _AN \ot_B M_A \rar{g \ot f} \dar[swap]{\eps} & _AN' \ot_B M'_A \dlar{\eps'} \\
      A & {}
    \end{tikzcd}
  \end{equation}
  commute.
\end{newdef}
If the algebras in question have the additional structure of a
symmetric Frobenius form $\lambda: A \to \K$, we would like to
formulate a compatibility condition between the Morita context and the
Frobenius forms.  We begin with the following two observations: if $A$
is an algebra, the map
\begin{equation}
  \label{eq:A/[A,A]-iso-A-ot-A}
  \begin{aligned}
    A/[A,A] &\to A \ot_{A \ot A^\op} A \\
    [a] & \mapsto a \ot 1
  \end{aligned}
\end{equation}
is an isomorphism of vector spaces, with inverse given by $a \ot b
\mapsto [ab]$.  Furthermore, if $B$ is another algebra, and
${(_BM_A},{ _AN_B}, \eps, \eta)$ is a Morita context between $A$ and
$B$, there is a canonical isomorphism of vector spaces
\begin{equation}
  \label{eq:tensor-product-switch}
  \begin{aligned}
    \tau: (N \ot_B M) \ot_{A \ot A^\op} (N \ot_B M) &\to (M \ot_A N) \ot_{B \ot B^\op} (M \ot_A N) \\
    n \ot m \ot n' \ot m' &\mapsto m \ot n' \ot m' \ot n .
  \end{aligned}
\end{equation}
Using the results above, we can formulate a compatibility condition
between Morita context and Frobenius forms, as in the following lemma.
\begin{lemma}
  \label{lem:A/[A,A]-B[B,B]-iso}
  Let $A$ and $B$ be two algebras, and let $( {_{B}M_{A}},
  {_{A}N_{B}},\eps,\eta)$ be a Morita context between $A$ and $B$.
  Then, there is a canonical isomorphism of vector spaces
  \begin{align}
    \begin{aligned}
      f: A / [A,A] & \to B /[B,B] \\
      [a] &\mapsto \sum_{i,j} \left[\eta^{-1}(m_j.a \ot n_i) \right]
    \end{aligned}
  \end{align}
  where $n_i$ and $m_j$ are defined by
  \begin{equation}
    \eps^{-1}(1_A)= \sum_{i,j} n_i \ot m_j \in N \ot_B M .
  \end{equation}
\end{lemma}
\begin{proof}
  Consider the following chain of isomorphisms:
  \begin{equation}
    \begin{aligned}
      f:   A/[A,A] &\cong A \ot_{A \ot A^\op} A && \text{(by equation \ref{eq:A/[A,A]-iso-A-ot-A})} \\
      & \cong (N \ot_B M) \ot_{A \ot A^\op} (N \ot_B M) && \text{(using $\eps \ot \eps$)} \\
      &\cong (M \ot_A N) \ot_{B \ot B^\op } (M \ot_A N) && \text{(by equation \ref{eq:tensor-product-switch})} \\
      &\cong B \ot_{B \ot B^\op} B && \text{(using $\eta \ot \eta$)} \\
      &\cong B/[B,B] &&\text{(by equation
        \ref{eq:A/[A,A]-iso-A-ot-A})}
    \end{aligned}
  \end{equation}
  Chasing through those isomorphisms, we can see that the map $f$ is
  given by
  \begin{equation}
    f([a]) =\sum_{i,j} \left[\eta^{-1}(m_j .a \ot  n_i ) \right]
  \end{equation}
  as claimed.
\end{proof}

The isomorphism $f$ described in Lemma \ref{lem:A/[A,A]-B[B,B]-iso}
allows to introduce the following relevant definition.
\begin{newdef}
  \label{def:morita-compatible}
  Let $(A,\lambda^A)$ and $(B,\lambda^B)$ be two symmetric Frobenius
  algebras, and let $( {_{B}M_{A}}, {_{A}N_{B}},\eps,\eta)$ be a
  Morita context between $A$ and $B$. Since the Frobenius algebras are
  symmetric, the Frobenius forms necessarily factor through $A/[A,A]$
  and $B/[B,B]$.  We call the Morita context \emph{compatible} with
  the Frobenius forms, if the diagram
  \begin{equation}
    \label{eq:morita-compatible}
    \begin{tikzcd}
      A/[A,A] \ar{dr}[swap]{\lambda^A} \ar{rr}{f} && B/[B,B] \ar{dl}{\lambda^B} \\
      & \K &
    \end{tikzcd}
  \end{equation}
  commutes.
\end{newdef}

\begin{remark}
  \label{rem:compatibility}
  The definition of compatible Morita context of \cite[Definition
  3.72]{schommerpries-classification} requires another compatibility
  condition on the coproduct of the unit of the Frobenius
  algebras. However, a calculation using proposition
  \ref{prop:morita-context-compatible-equivalence} shows that the
  condition of \cite{schommerpries-classification} is already implied
  by our condition on Frobenius form of definition
  \ref{def:morita-compatible}; thus the two definitions of compatible
  Morita context do coincide.
\end{remark}

For later use, we give a very explicit way of expressing the
compatibility condition between Morita context and Frobenius forms:
if $(A, \lambda^A)$ and $(B,\lambda^B)$ are two finite-dimensional
semi-simple symmetric Frobenius algebras over an algebraically closed field $\K$, and $({_BM_A}, {_BN_A
},\eps, \eta)$ is a Morita context between them, the algebras $A$ and
$B$ are isomorphic to direct sums of matrix algebras by
Artin-Wedderburn:
\begin{equation}
  A  \cong \bigoplus_{i=1}^{r} M_{d_i}(\K), \qquad \text{and} \qquad
  B  \cong \bigoplus_{j=1}^{r} M_{n_j}(\K).
\end{equation}

By Theorem 3.3.1 of \cite{eghlsvy11}, the simple modules $(S_1,
\ldots, S_r)$ of $A$ and the simple modules $(T_1, \ldots, T_r)$ of
$B$ are given by $S_i:=\K^{d_i}$ and $T_i:= \K^{n_i}$, and every
module is a direct sum of copies of those. Since simple
finite-dimensional representations of $A \ot_\K B^\op$ are given by
tensor products of simple representations of $A$ and $B^\op$ by
Theorem 3.10.2 of \cite{eghlsvy11}, the most general form of $_BM_A$
and $_AN_B$ is given by
\begin{equation}
  \begin{aligned}
    _BM_A:&= \bigoplus_{i,j=1}^{r} \alpha_{ij} \, T_i \ot_\K S_j   \\
    _AN_B:&= \bigoplus_{k,l=1}^{r} \beta_{kl} \, S_k \ot_\K T_l
  \end{aligned}
\end{equation}
where $\alpha_{ij}$ and $\beta_{kl}$ are multiplicities.  First, we
show that the multiplicities are trivial:
\begin{lemma}
  \label{lem:most-general-bimodule}
  In the situation as above, the multiplicities are trivial after a
  possible reordering of the simple modules: $\alpha_{ij}
  =\delta_{ij}=\beta_{ij}$ and the two bimodules $M$ and $N$ are
  actually given by
  \begin{equation}
    \label{eq:most-general-bimodule}
    \begin{aligned}
      _BM_A&= \bigoplus_{i=1}^{r}  T_i \ot_\K S_i   \\
      _AN_B&= \bigoplus_{j=1}^{r} S_j \ot_\K T_j .
    \end{aligned}
  \end{equation}
\end{lemma}
\begin{proof}
  Suppose for a contradiction that there is a term of the form $(T_i \oplus
  T_j) \ot S_k$ in the direct sum decomposition of $M$. Let $f: T_i
  \to T_j$ be a non-trivial linear map, and define $\varphi \in \End_A
  ((T_i \oplus T_j) \ot S_k)$ by setting $\varphi((t_i +t_j) \ot
  s_k):=f(t_i) \ot s_k$. The $A$-module map $\varphi$ induces an
  $A$-module endomorphism on all of $_AM_B$ by extending $\varphi$
  with zero on the rest of the direct summands. Since $\End_A(_BM_A)
  \cong B$ as algebras by Theorem 3.5 of \cite{bass}, the endomorphism
  $\varphi$ must come from left multiplication, which cannot be true
  for an arbitrary linear map $f$.  This shows that the bimodule $M$
  is given as claimed in equation
  \eqref{eq:most-general-bimodule}. The statement for the other
  bimodule $N$ follows analogously.
\end{proof}
Lemma \ref{lem:most-general-bimodule} shows how the bimodules
underlying a Morita context of semi-simple algebras look like. Next,
we consider the Frobenius structure.
\begin{lemma}[{\cite[Lemma 2.2.11]{kock2004frobenius}}]
  \label{lem:symmetric-frobenius-form-unique-up-to-multiple}
  Let $(A, \lambda)$ be a symmetric Frobenius algebra. Then, every
  other symmetric Frobenius form on $A$ is given by multiplying the
  Frobenius form with a central invertible element of $A$.
\end{lemma}
By Lemma \ref{lem:symmetric-frobenius-form-unique-up-to-multiple}, we
conclude that the Frobenius forms on the two semi-simple algebras $A$
and $B$ are given by
\begin{equation}
  \label{eq:frob-form-ssi-alg}
  \lambda^A = \bigoplus_{i=1}^{r} \lambda^A_i \, \tr_{M_{d_i}(\K)} \qquad \text{and} \qquad 
  \lambda^B =\bigoplus_{i=1}^{r} \lambda^B_i \tr_{M_{n_i}(\K)}
\end{equation}
where $\lambda^A_i$ and $\lambda^B_i$ are non-zero scalars. We can now
state the following proposition, which will be used in the proof of
corollary \ref{cor:frob-fixed-points}.
\begin{prop}
  \label{prop:morita-context-compatible-equivalence}
  Let $(A, \lambda^A)$ and $(B, \lambda^B)$ be two finite-dimensional,
  semi-simple symmetric Frobenius algebras and suppose that
  $\cM:=(M,N,\eps,\eta)$ is a Morita context between them. Let
  $\lambda_i^A$ and $\lambda_j^B$ be as in equation
  \eqref{eq:frob-form-ssi-alg}, and define two invertible central elements 
  \begin{equation}
    \begin{aligned}
      a:&=(\lambda_1^A, \ldots, \lambda_r^A) \in \K^r \cong Z(A) \\
      b:&=(\lambda_1^B, \ldots, \lambda_r^B) \in \K^r \cong Z(B)
    \end{aligned}
  \end{equation}
  Then, the following are equivalent:
  \begin{enumerate}
  \item The Morita context $\cM$ is compatible with the Frobenius forms in
    the sense of definition \ref{def:morita-compatible}.
  \item We have $m.a=b.m$ for all $m \in {_BM_A}$ and
    $n.b^{-1}=a^{-1}.n$ for all $n \in {_AN_B}$.
  \item For every $i=1, \ldots, r$, we have that
    $\lambda_i^A=\lambda_i^B$.
  \end{enumerate}
\end{prop}
\begin{proof}
  With the form of $M$ and $N$ determined by equation
  \eqref{eq:most-general-bimodule}, we see that the only isomorphisms
  of bimodules $\eps: N \ot_B M \to A$ and $\eta: B \to M \ot_A N$
  must be given by multiples of the identity matrix on each direct
  summand:
  \begin{equation}
    \begin{aligned}
      \eps: N \ot_A M \cong \bigoplus_{i=1}^r M(d_i \times d_i, \K) & \to  \bigoplus_{i=1}^r M(d_i \times d_i, \K) =A \\
      \sum_{i=1}^r M_i & \mapsto \sum_{i=1}^r \eps_i M_i
    \end{aligned}
  \end{equation}
  Similarly, $\eta$ is given by
  \begin{equation}
    \begin{aligned}
      \eta:  B = \bigoplus_{i=1}^r M(n_i \times n_i , \K) &\mapsto M \ot_A B \cong \bigoplus_{i=1}^r M(n_i \times n_i , \K) \\
      \sum_{i=1}^r M_i & \mapsto \sum_{i=1}^r \eta_i M_i
    \end{aligned}
  \end{equation}
  Here, $\eps_i$ and $\eta_i$ are non-zero scalars. The condition that
  this data should be a Morita context then demands that $\eps_i =
  \eta_i$, as a short calculation in a basis confirms. By calculating
  the action of the elements $a$ and $b$ defined above in a basis, we
  see that conditions $(2)$ and $(3)$ of the above proposition are
  equivalent.

  Next, we show that $(1)$ and $(3)$ are equivalent. In order to see
  when the Morita context is compatible with the Frobenius forms, we
  calculate the map $f:A/[A,A] \to B/[B,B]$ from equation
  \eqref{eq:morita-compatible}. One way to do this is to notice that
  $[A,A]$ consists precisely of trace-zero matrices
  (cf. \cite{albert1957}); thus
  \begin{equation}
    \begin{aligned}
      A/[A,A] &\to \K^r \\
      [ A_1 \oplus A_2 \oplus \cdots \oplus A_r]& \mapsto \left(
        \tr(A_1), \cdots, \tr(A_r) \right)
    \end{aligned}
  \end{equation}
  is an isomorphism of vector spaces.  Using this identification, we
  see that the map $f$ is given by
  \begin{equation}
    \begin{aligned}
      f:A/[A,A] & \to B/[B,B] \\
      [ A_1 \oplus A_2 \oplus \cdots \oplus A_r] & \mapsto
      \bigoplus_{i=1}^r \tr_{M_{d_i}}(A_i) \left[ E^{(n_i \times
          n_i)}_{11} \right]
    \end{aligned}
  \end{equation}
  Note that this map is independent of the scalars $\eps_i$ and
  $\eta_i$ coming from the Morita context. Now, the two Frobenius
  algebras $A$ and $B$ are Morita equivalent via a compatible Morita
  context if and only if the diagram in equation
  \eqref{eq:morita-compatible} commutes. This is the case if and only
  if $\lambda^A_i = \lambda^B_i$ for all $i$, as a straightforward
  calculation in a basis shows.
\end{proof}
Having established how compatible Morita contexts between semi-simple
algebras over an algebraic closed field look like, we arrive at
following definition.
\begin{newdef}
  \label{def:frob-bicategory}
  Let $\K$ be an algebraically closed field.  Let $\Frob$ be the
  bicategory where
  \begin{itemize}
  \item objects are given by finite-dimensional, semisimple, symmetric
    Frobenius $\K$-algebras,
  \item 1-morphisms are given by compatible Morita contexts, as in
    definition \ref{def:morita-compatible},
  \item 2-morphisms are given by isomorphisms of Morita contexts.
  \end{itemize}
  Note that $\Frob$ has got the structure of a symmetric monoidal
  bigroupoid, where the monoidal product is given by the tensor
  product over the ground field, which is the monoidal unit.
\end{newdef}
The bicategory $\Frob$ will be relevant for the remainder of the
paper, due to the following theorem.
\begin{theorem}[Oriented version of the Cobordism Hypothesis,
  \cite{schommerpries-classification}]
  The weak 2-functor
  \begin{equation}
    \begin{aligned}
      \Fun_\ot(\Cob_{2,1,0}^\ori, \Alg_2) & \to \Frob \\
      Z &\mapsto Z(+)
    \end{aligned}
  \end{equation}
  is an equivalence of bicategories.
\end{theorem}

\section{Group actions on bicategories and their homotopy fixed
  points}
\label{sec:group-actions-bicat}
For a group $G$, we denote with $\B G$ the category with one object
and $G$ as morphisms. Similarly, if $\cC$ is a monoidal category,
$\B\cC$ will denote the bicategory with one object and $\cC$ as
endomorphism category of this object.  Furthermore, we denote by
$\ubar G$ the discrete monoidal category associated to $G$, i.e. the
category with the elements of $G$ as objects, only identity morphisms,
and monoidal product given by group multiplication.

Recall that an action of a group $G$ on a set $X$ is a group
homomorphism $\rho: G \to \Aut(X)$. The set of fixed points $X^G$ is
then defined as the set of all elements of $X$ which are invariant
under the action. In equivalent, but more categorical terms, a
$G$-action on a set $X$ can be defined to be a functor $\rho:\B{G} \to
\Set$ which sends the one object of the category $\B{G}$ to the set
$X$.

If $\Delta : \B{G} \to \Set$ is the constant functor sending the one
object of $\B G$ to the set with one element, one can check that the
set of fixed points $X^G$ stands in bijection to the set of natural
transformations from the constant functor $\Delta$ to $\rho$, which is
exactly the limit of the functor $\rho$. Thus, we have bijections of
sets
\begin{equation}
  \label{eq:fixed-points-are-natural-trafos}
  X^G \cong \lim_{*//G} \rho \cong \Nat(\Delta ,\rho).
\end{equation}
\begin{remark}
  \label{rem:action-as-monoidal-functor}
  A further equivalent way of providing a $G$-action on a set $X$ is
  by giving a monoidal functor $\rho: \ubar{G} \to \ubar{\Aut}(X)$,
  where we regard both $G$ and $\Aut(X)$ as categories with only
  identity morphisms. This definition however does not allow us to
  express the set of homotopy fixed points in a nice categorical way
  as in equation \eqref{eq:fixed-points-are-natural-trafos}, and thus
  turns out to be less useful for our purposes.
\end{remark}
Categorifying the notion of a $G$-action on a set yields the
definition of a discrete group acting on a category:
\begin{newdef}
  \label{def:action-on-cat}
  Let $G$ be a discrete group and let $\cC$ be a category. Let
  $\B{\ubar G}$ be the 2-category with one object and ${\ubar G}$ as
  the category of endomorphisms of the single object. A $G$-action on
  $\cC$ is defined to be a weak 2-functor $\rho: B{\ubar G} \to \Cat$
  with $\rho(*)=\cC$.
\end{newdef}

Note that just as in remark \ref{rem:action-as-monoidal-functor}, we could have avoided the language of 2-categories and have defined a $G$-action on a category $\cC$ to be a monoidal
functor $\rho: \ubar{G} \to \Aut(\cC)$.

Next, we would like to define the homotopy fixed point category of
this action to be a suitable limit of the action, just as in equation
\eqref{eq:fixed-points-are-natural-trafos}. The appropriate notion of
a limit of a weak 2-functor with values in a bicategory appears in the
literature as a \emph{pseudo-limit} or \emph{indexed limit}, which we
will simply denote by $\lim$. We will only consider limits indexed by
the constant functor. For background, we refer the reader to
\cite{lack-2-cat-companion}, \cite{kelly-elementary},
\cite{street-fibs-bicats} and \cite{street-fibs-bicats-correction}.

We are now in the position to introduce the following definition:
\begin{newdef}
  Let $G$ be a discrete group, let $\cC$ be a category, and let $\rho:
  \B{\ubar G} \to \Cat$ be a $G$-action on $\cC$. Then, the category
  of homotopy fixed points $\cC^G$ is defined to be the pseudo-limit
  of $\rho$.
\end{newdef}
Just as in the 1-categorical case in equation
\eqref{eq:fixed-points-are-natural-trafos}, it is shown in
\cite{kelly-elementary} that the limit of any weak 2-functor with
values in $\Cat$ is equivalent to the category of pseudo-natural
transformations and modifications $\Nat(\Delta, \rho)$ . Hence, we
have an equivalence of categories
\begin{equation}
  \label{eq:fixed-points-cat-natural-trafos}
  \cC^G \cong \lim \rho \cong \Nat(\Delta, \rho).
\end{equation}
Here, $\Delta: \B{\ubar G} \to \Cat$ is the constant functor sending
the one object of $\B{\ubar G}$ to the terminal category with one
object and only the identity morphism.  By spelling out definitions,
one sees:
\begin{remark}
  Let $\rho: \B{\ubar G} \to \Cat$ be a $G$-action on a category
  $\cC$, and suppose that $\rho(e)=\id_\cC$, i.e. the action respects
  the unit strictly. Then, the homotopy fixed point category $\cC^G$
  is equivalent to the \enquote{equivariantization} introduced in
  \cite[Definition 2.7.2]{egno-tensor-book}.
\end{remark}

\subsection{G-actions on bicategories} Next, we would like to step up
the categorical ladder once more, and define an action of a group $G$
on a bicategory. Moreover, we would also like to account for the case
where our group is equipped with a topology. This will be done by
considering the fundamental 2-groupoid of $G$, referring the reader to
\cite{hkk-fundamental} for additional details.

\begin{newdef}
  Let $G$ be a topological group. The fundamental 2-groupoid of $G$ is
  the monoidal bicategory $\Pi_2(G)$ where
  \begin{itemize}
  \item objects are given by points of $G$,
  \item 1-morphisms are given by paths between points,
  \item 2-morphisms are given by homotopy classes of homotopies
    between paths, called \emph{2-tracks}.
  \end{itemize}
  The monoidal product of $\Pi_2(G)$ is given by the group
  multiplication on objects, by pointwise multiplication of paths on
  1-morphisms, and by pointwise multiplication of 2-tracks on
  2-morphisms. Notice that this monoidal product is associative on the
  nose, and all other monoidal structure like associators and unitors
  can be chosen to be trivial.
\end{newdef}
We are now ready to give a definition of a $G$-action on a
bicategory. Although the definition we give uses the language of
tricategories as defined in \cite{gps95} or \cite{gur07}, we provide a
bicategorical description in Remark \ref{rem:unpacking-action}.
\begin{newdef}
  \label{def:action-on-bicat}
  Let $G$ be a topological group, and let $\cC$ be a bicategory. A
  $G$-action on $\cC$ is defined to be a trifunctor
  \begin{equation}
    \rho: \B{\Pi_2(G)} \to \Bicat
  \end{equation}
  with $\rho(*)=\cC$.  Here, $\B{\Pi_2(G)}$ is the tricategory with
  one object and with $\Pi_2(G)$ as endomorphism-bicategory, and
  $\Bicat$ is the tricategory of bicategories.
\end{newdef}

\begin{remark}
  If $\cC$ is a bicategory, let $\Aut(\cC)$ be the bicategory
  consisting of auto-equivalences of bicategories of $\cC$,
  pseudo-natural isomorphisms and invertible modifications. Observe
  that $\Aut(\cC)$ has the structure of a monoidal bicategory, where
  the monoidal product is given by composition. Since there are two
  ways to define the horizontal composition of pseudo-natural
  transformation, which are \emph{not} equal to each other, there are
  actually two monoidal structures on $\Aut(\cC)$. It turns out that
  these two monoidal structures are equivalent; see \cite[Section
  5]{gps95} for a discussion in the language of
  tricategories. 

  With either monoidal structure of $\Aut(\cC)$ chosen, note that as
  in Remark \ref{rem:action-as-monoidal-functor} we could equivalently
  have defined a $G$-action on a bicategory $\cC$ to be a weak
  monoidal 2-functor $ \rho: \Pi_2(G) \to \Aut(\cC)$.
\end{remark}

Since we will only consider trivial actions in this paper, the hasty
reader may wish to skip the next remark, in which the definition of a
$G$-action on a bicategory is unpacked. We will, however use the
notation introduced here in our explicit description of homotopy fixed
points in remark \ref{rem:unpacking-fixed-point}.
\begin{remark}[Unpacking Definition \ref{def:action-on-bicat}]
  \label{rem:unpacking-action}
  Unpacking the definition of a weak monoidal 2-functor $\rho:
  \Pi_2(G) \to \Aut(\cC)$, as for instance in \cite[Definition
  2.5]{schommerpries-classification}, or equivalently of a trifunctor
  $\rho: \B{\Pi_2(G)} \to \Bicat$, as in \cite[Definition 3.1]{gps95},
  shows that a $G$-action on a bicategory $\cC$ consists of the
  following data and conditions:
  \begin{itemize}
  \item For each group element $g \in G$, an equivalence of
    bicategories $F_g:=\rho(g) : \cC \to
    \cC$,
  \item For each path $\gamma : g \to h$ between two group elements,
    the action assigns a pseudo-natural isomorphism $\rho(\gamma): F_g
    \to F_h$, 
  \item For each 2-track $m : \gamma \to \gamma'$, the action assigns
    an invertible modification $\rho(m) : \rho(\gamma) \to
    \rho(\gamma')$.
  \item There is additional data making $\rho$ into a weak 2-functor,
    namely: if $\gamma_1 : g \to h$ and $\gamma_2 : h \to k$ are paths
    in $G$, we obtain invertible modifications
    \begin{equation}
      \phi_{\gamma_2 \gamma_1}: \rho(\gamma_2) \circ \rho(\gamma_1) \to \rho(\gamma_2 \circ \gamma_1) 
    \end{equation}
  \item Furthermore, for every $g \in G$ there is an invertible
    modification $\phi_g: \id_{F_g} \to \rho(\id_g)$ between the
    identity endotransformation on $F_g$ and the value of $\rho$ on
    the constant path $\id_g$.

    There are three compatibility conditions for this data: one
    condition making $\phi_{\gamma_2, \gamma_1}$ compatible with the
    associators of $\Pi_2(G)$ and $\Aut(\cC)$, and two conditions with
    respect to the left and right unitors of $\Pi_2(G)$ and
    $\Aut(\cC)$.
  \item Finally, there are data and conditions for $\rho$ to be
    monoidal. These are:
    \begin{itemize}
    \item A pseudo-natural isomorphism
      \begin{equation}
        \chi: \rho(g) \ot \rho(h) \to
        \rho(g \ot h)
      \end{equation}
    \item A pseudo-natural isomorphism
      \begin{equation}
        \label{eq:iota-action}
        \iota: \id_\cC \to F_e
      \end{equation}
    \item For each triple $(g,h,k)$ of group elements, an invertible
      modification $\omega_{ghk}$ in the diagram
      \begin{equation}
        \label{eq:omega}
        \vcenter{\hbox{\includegraphics{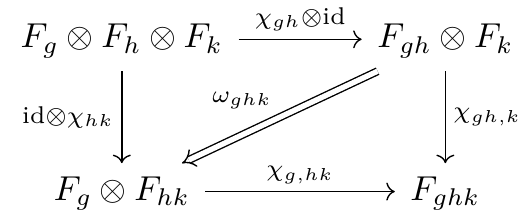}}}
      \end{equation}
    \item An invertible modification $\gamma$ in the triangle below
      \begin{equation}
        \label{eq:gamma}
        \vcenter{\hbox{\includegraphics{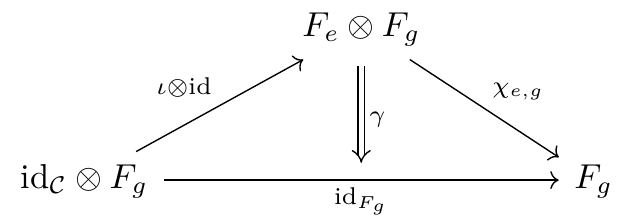}}}
      \end{equation}
    \item Another invertible modification $\delta$ in the triangle
      \begin{equation}
        \label{eq:delta}
        \vcenter{\hbox{\includegraphics{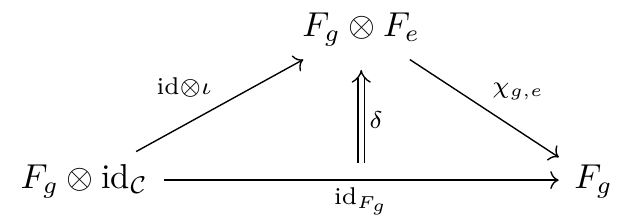}}}
      \end{equation}
    \end{itemize}
  \end{itemize}
  The data $(\rho,\chi, \iota, \omega, \gamma, \delta)$ then has to
  obey equations (HTA1) and (HTA2) in \cite[p. 17]{gps95}.
\end{remark}

Just as in the case of a group action on a set and a group action on a
category, we would like to define the bicategory of homotopy fixed
points of a group action on a bicategory as a suitable limit. However,
the theory of trilimits is not very well established. Therefore we
will take the description of homotopy fixed points as natural
transformations as in equation
\eqref{eq:fixed-points-are-natural-trafos} as a definition, and define
homotopy fixed points of a group action on a bicategory as the
bicategory of pseudo-natural transformations between the constant
functor and the action.
\begin{newdef}
  \label{def:fixed-point}
  Let $G$ be a topological group and $\cC$ a bicategory.  Let
  \begin{equation}
    \rho: \B{\Pi_2(G)} \to \Bicat
  \end{equation}
  be a $G$-action on $\cC$.  The bicategory of homotopy fixed points
  $\cC^G$ is defined to be
  \begin{equation}
    \cC^G:=\Nat(\Delta, \rho)
  \end{equation}
  Here, $\Delta$ is the constant functor which sends the one object of
  $\B{\Pi_2(G)}$ to the terminal bicategory with one object, only the
  identity 1-morphism and only identity 2-morphism. The bicategory
  $\Nat(\Delta, \rho)$ then has objects given by tritransformations
  $\Delta \to \rho$, 1-morphisms are given by modifications, and
  2-morphisms are given by perturbations.
\end{newdef}
\begin{remark}
  The notion of the \enquote{equivariantization} of a strict 2-monad
  on a 2-category has already appeared in \cite[Section
  6.1]{mn14}. Note that definition \ref{def:fixed-point} is more
  general than the definition of \cite{mn14}, in which some
  modifications have been assumed to be trivial.
\end{remark}
\begin{remark}
  In principle, even higher-categorical definitions are possible: for
  instance in \cite{fiva15} a homotopy fixed point of a higher
  character $\rho$ of an $\infty$-group is defined to be a (lax)
  morphism of $\infty$-functors $\Delta \to \rho$.
\end{remark}
\begin{remark}[Unpacking objects of $\cC^G$]
  \label{rem:unpacking-fixed-point}
  Since unpacking the definition of homotopy fixed points is not
  entirely trivial, we spell it out explicitly in the subsequent
  remarks, following \cite[Definition 3.3]{gps95}. In the language of
  bicategories, a homotopy fixed point consists of:
  \begin{itemize}
  \item an object $c$ of $\cC$,
  \item a pseudo-natural equivalence
    \begin{equation}
      \vcenter{\hbox{\includegraphics{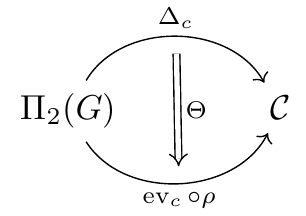}}} 
    \end{equation}
    where $\Delta_c$ is the constant functor which sends every object to $c \in \cC$, and $\ev_c$ is the evaluation at the object $c$.\\
    In components, the pseudo-natural transformation $ \Theta$
    consists of the following:
    \begin{itemize}
    \item for every group element $g \in G$, a 1-equivalence in $\cC$
      \begin{equation}
        \Theta_g:c \to F_g(c)
      \end{equation}
    \item and for each path $\gamma: g \to h$, an invertible
      2-morphism $\Theta_\gamma$ in the diagram 
      \begin{equation}
        \vcenter{\hbox{\includegraphics{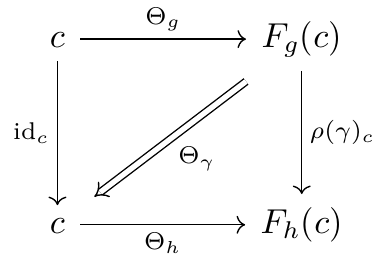}}} 
      \end{equation}
      which is natural with respect to 2-tracks.
    \end{itemize}
  \item an invertible modification $\Pi$ in the diagram 
    \begin{equation}
      \vcenter{\hbox{\includegraphics{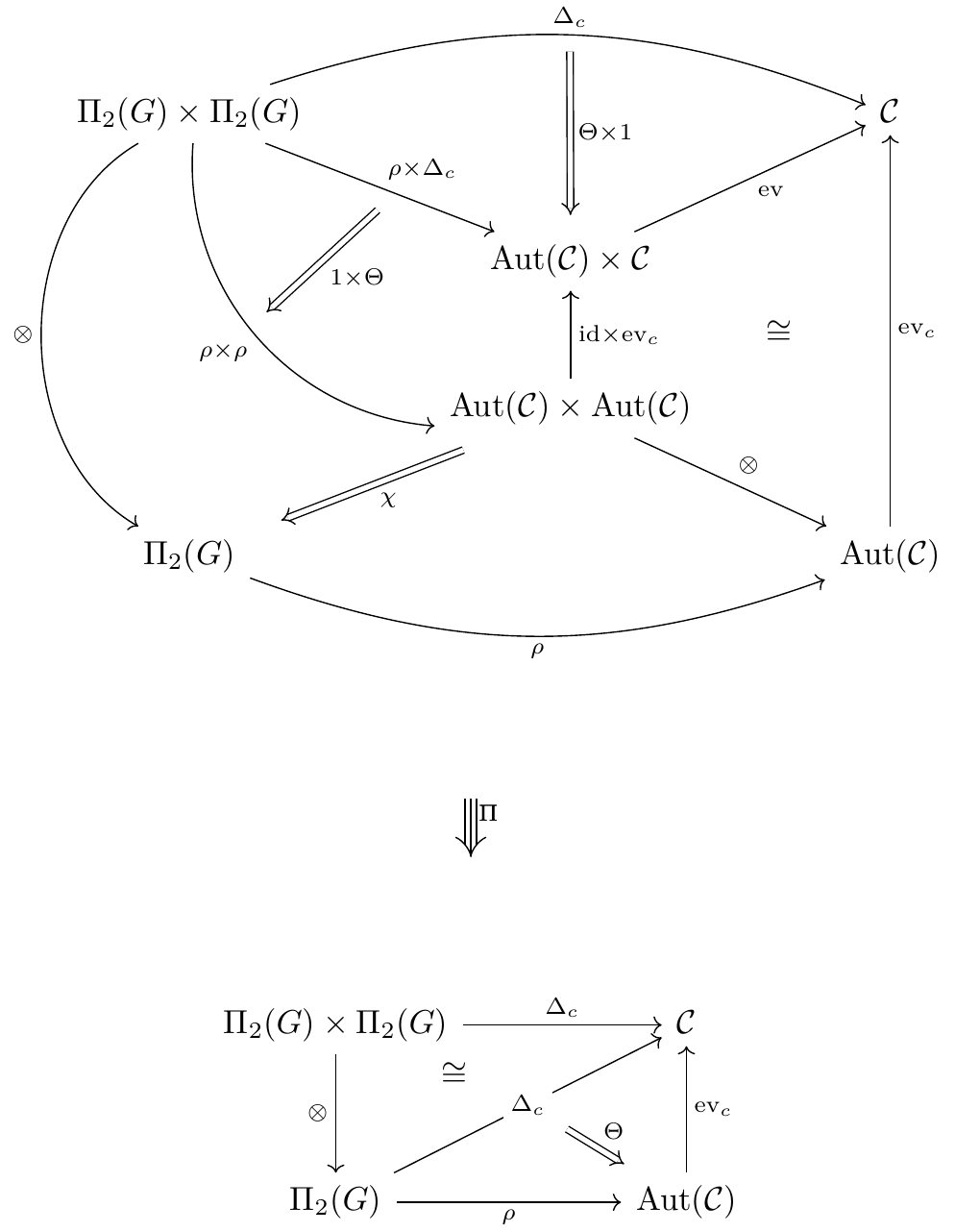}}}
    \end{equation} \\
    which in components means that for every tuple of group
    elements $(g,h)$ we have an invertible 2-morphism $\Pi_{gh}$ in the
    diagram 
    \begin{equation}
      \label{eq:pi}
      \vcenter{\hbox{\includegraphics{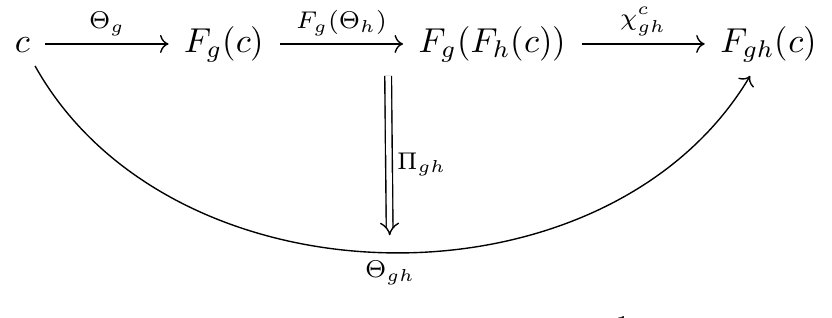}}}
    \end{equation}
  \item for the unital structure, another invertible modification $M$,
    which only has the component given in the diagram 
    \begin{equation}
      \vcenter{\hbox{\includegraphics{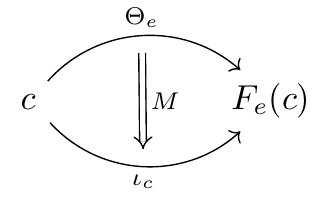}}}
    \end{equation}
  \end{itemize}
  with $\iota$ as in equation \eqref{eq:iota-action}.  The data $(c,
  \Theta, \Pi, M)$ of a homotopy fixed point then has to obey the
  following three conditions. Using the equation in
  \cite[p.21-22]{gps95} we find the condition
  \begin{equation}
    \label{eq:fixed-point-condition1}
    \vcenter{\hbox{\includegraphics[width=\textwidth]{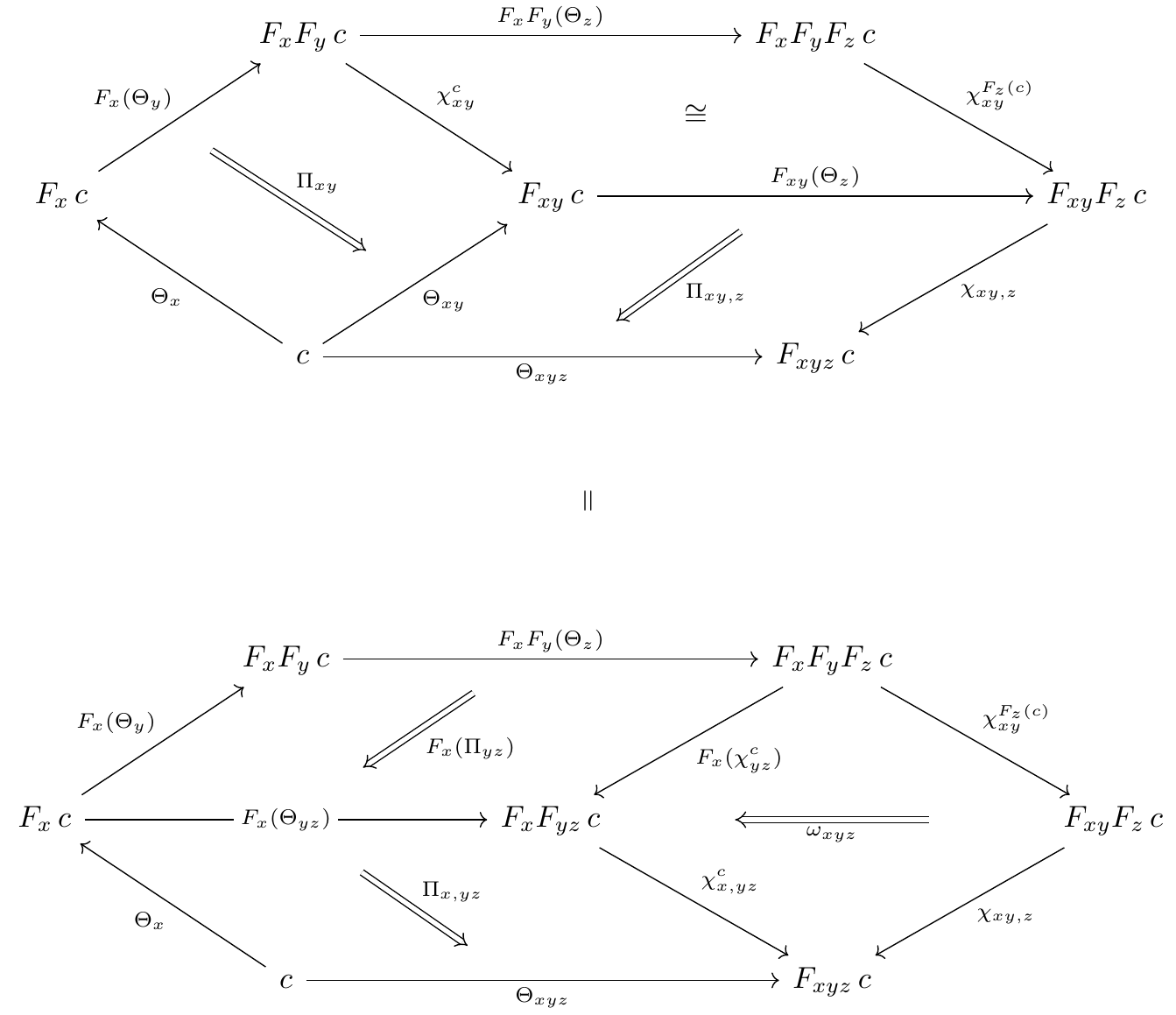}}}
  \end{equation}
  whereas the equation on p.23 of \cite{gps95} demands that we have
  \begin{equation}
    \label{eq:fixed-point-condition2}
    \vcenter{\hbox{\includegraphics{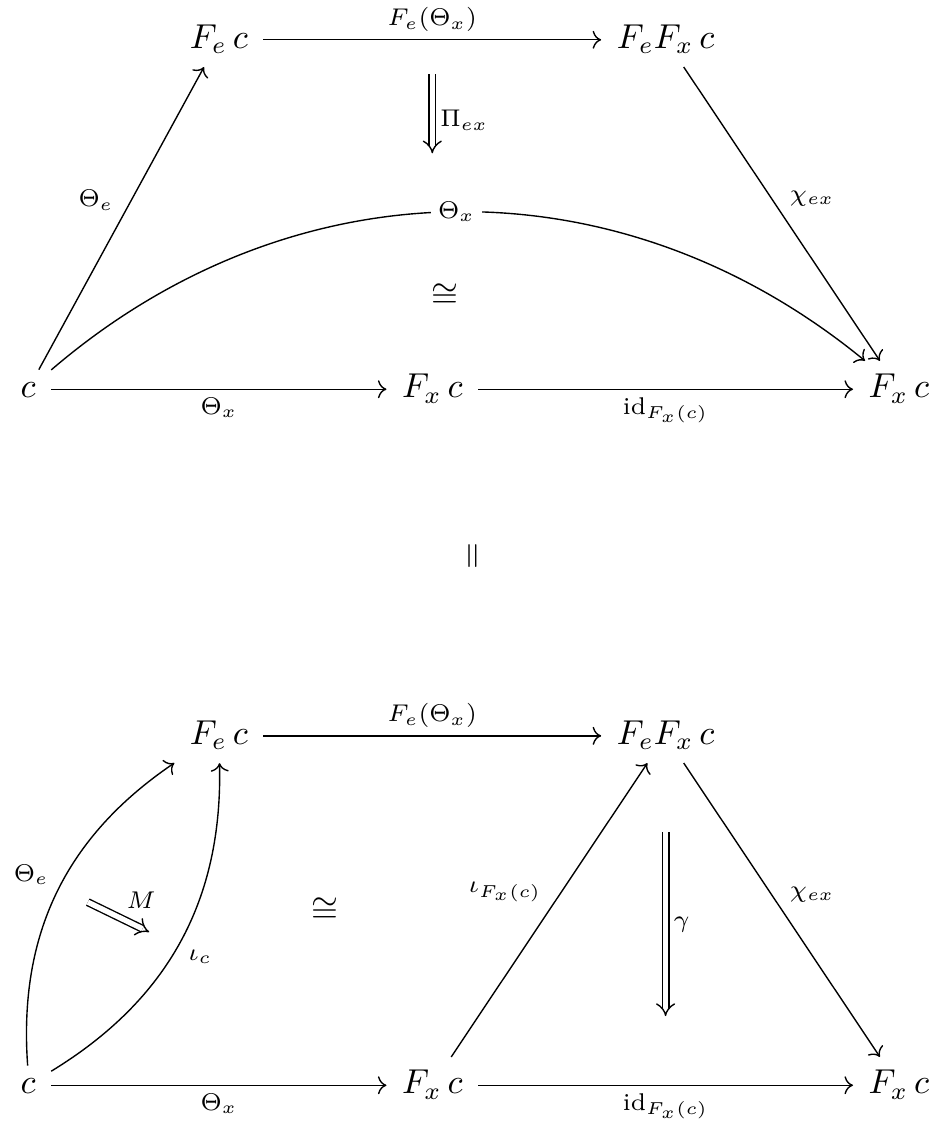}}}
  \end{equation}
  and finally the equation on p.25 of \cite{gps95} demands that
  \begin{equation}
    \label{eq:fixed-point-condition3}
    \vcenter{ \hbox{\includegraphics{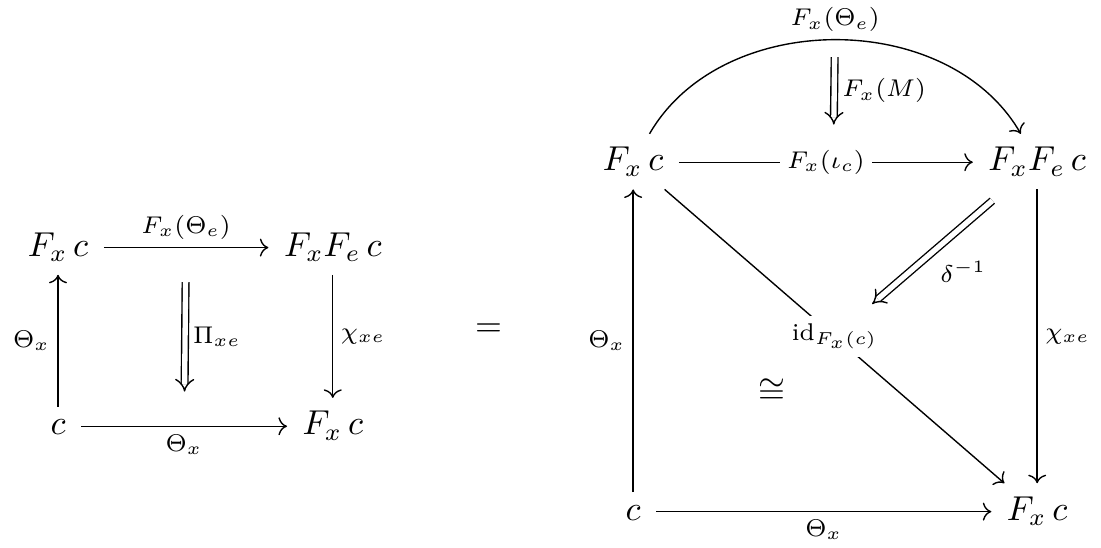}}}
  \end{equation}
\end{remark}

\begin{remark}
  Suppose that $(c, \Theta, \Pi, M)$ and $(c', \Theta', \Pi', M')$ are
  homotopy fixed points. A 1-morphism between these homotopy fixed
  points consists of a trimodification. In detail, this means:
  \begin{itemize}
  \item A 1-morphism $f: c \to c'$,
  \item An invertible modification $m$ in the diagram 
    \begin{equation}
      \label{eq:m-g-modification}
      \vcenter{\hbox{\includegraphics{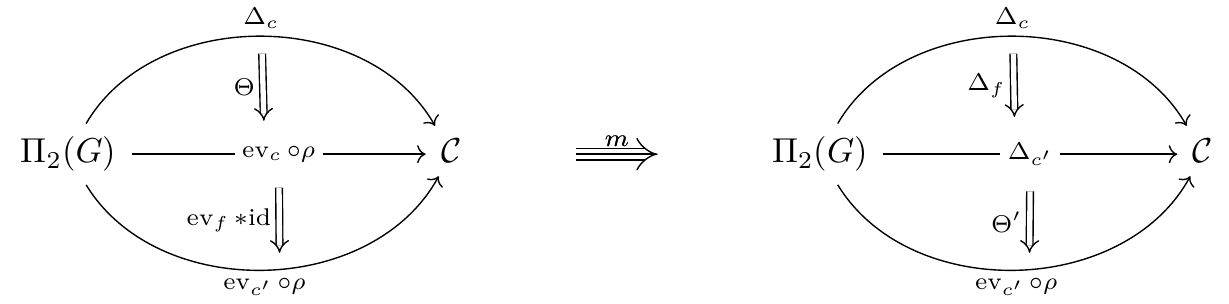}}}
    \end{equation}
    In components, $m_g$ is given by
    \begin{equation}
      \vcenter{\hbox{\includegraphics{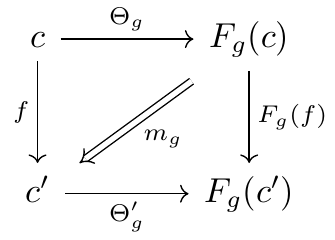}}}
    \end{equation}
    The data $(f,m)$ of a 1-morphism of homotopy fixed points has to
    satisfy the following two equations as on p.25 and p. 26 of
    \cite{gps95}:
  \end{itemize}
  \begin{equation}
    \label{eq:1-morphism-fixed-point-condition1}
    \vcenter{\hbox{\includegraphics{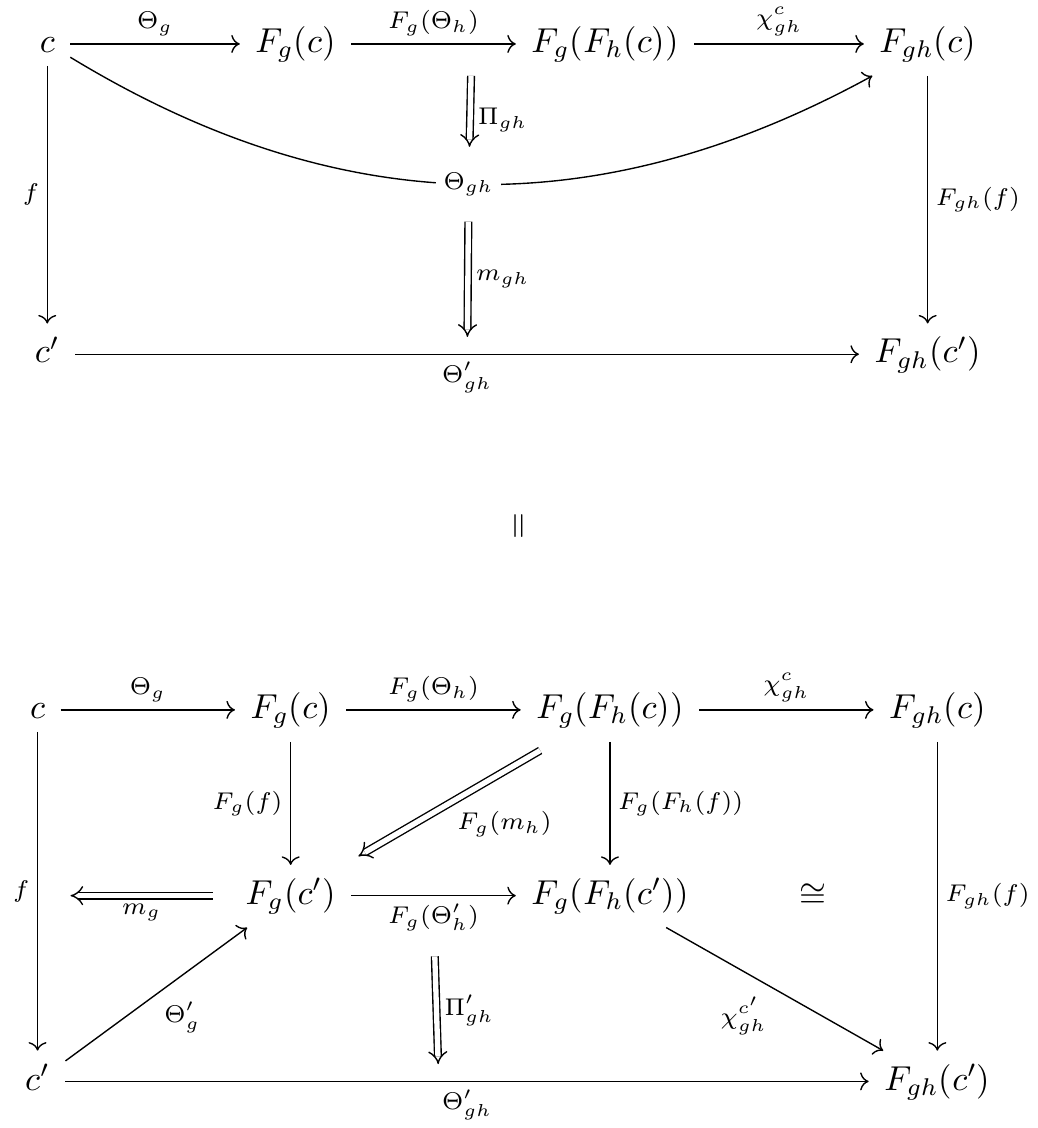}}}
  \end{equation}
  whereas the second equation reads
  \begin{equation}
    \label{eq:1-morphism-fixed-point-condition2}
    \vcenter{\hbox{\includegraphics{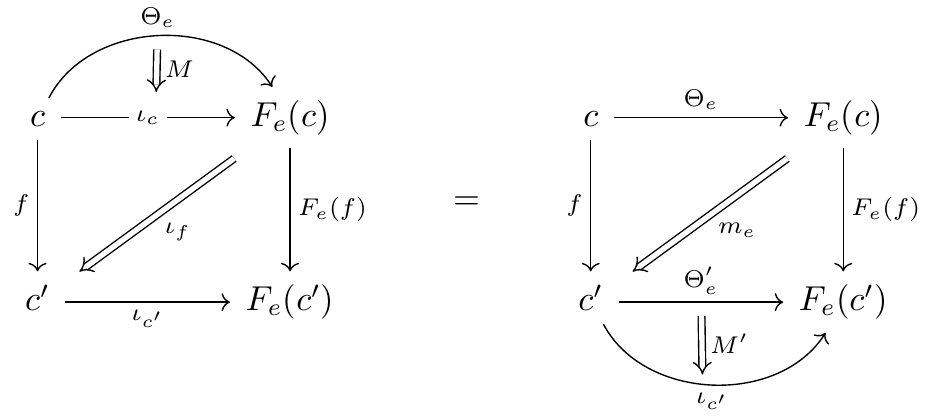}}}
  \end{equation}
\end{remark}

\begin{remark}
  The condition saying that $m$, as introduced in equation
  \eqref{eq:m-g-modification}, is a modification will be vital for the
  proof of Theorem \ref{thm:so2-fixed-point-bicat} and states that for
  every path $\gamma: g \to h$ in $G$, we must have the following
  equality of 2-morphisms in the two diagrams:
  \begin{equation}
    \label{eq:1-morphism-fixed-point-modification-condition}
    \vcenter{\hbox{\includegraphics{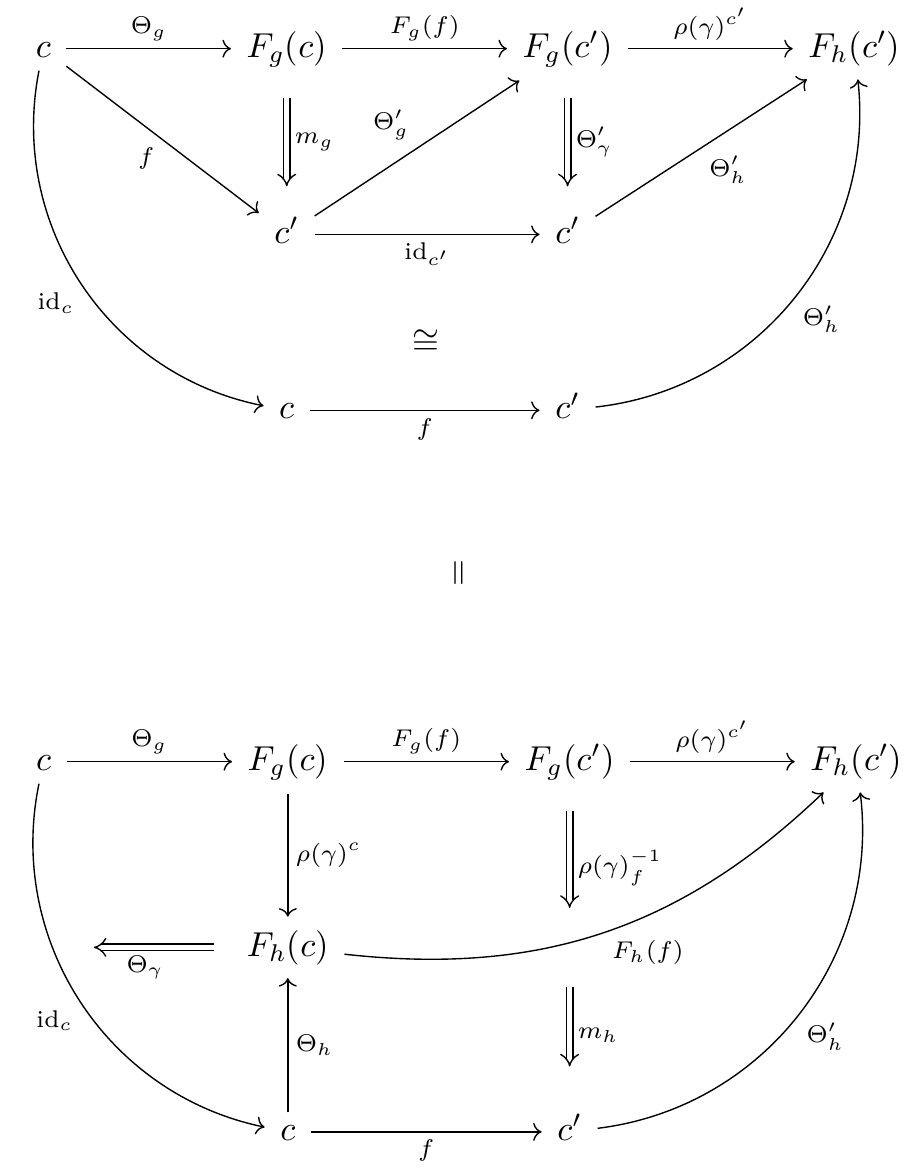}}}
  \end{equation}
\end{remark}

Next, we come to 2-morphisms of the bicategory $\cC^G$ of homotopy fixed points:
\begin{remark}
  Let $(f,m), (\xi,n): (c, \Theta, \Pi, M) \to (c', \Theta', \Pi',
  M')$ be two 1-morphisms of homotopy fixed points. A 2-morphism of
  homotopy fixed points consists of a perturbation between those
  trimodifications. In detail, a 2-morphism of homotopy fixed points
  consists of a 2-morphism $\alpha: f \to \xi$ in $\cC$, so that
  \begin{equation}
    \label{eq:2-morphism-fixed-point-condition}
    \vcenter{\hbox{\includegraphics{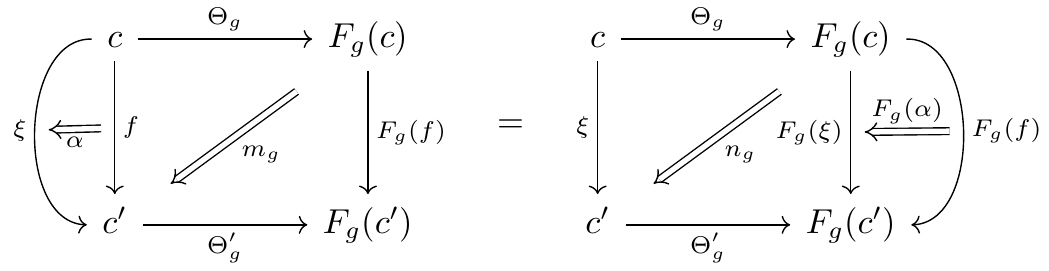}}}
  \end{equation}
\end{remark}
Let us give an example of a group action on bicategories and its
homotopy fixed points:
\begin{ex}
  Let $G$ be a discrete group, and let $\cC$ be any
  bicategory. Suppose $\rho:\Pi_2(G) \to \Aut(\cC)$ is the trivial
  $G$-action. Then, by remark \ref{rem:unpacking-fixed-point} a
  homotopy fixed point, i.e. an object of $\cC^G$ consists of
  \begin{itemize}
  \item an object $c$ of $\cC$,
  \item a 1-equivalence $\Theta_g :c \to c$ for every $g \in G$,
  \item a 2-isomorphism $\Pi_{gh}: \Theta_h \circ \Theta_g \to
    \Theta_{gh}$,
  \item a 2-isomorphism $M:\Theta_e \to \id_c$.
  \end{itemize}
  This is exactly the same data as a functor $\B{\ubar G} \to \cC$,
  where $\B{\ubar G}$ is the bicategory with one object, $G$ as
  morphisms, and only identity 2-morphisms.  Extending this analysis
  to 1- and 2-morphisms of homotopy fixed points shows that we have an
  equivalence of bicategories
  \begin{equation}
    \cC^G \cong \Fun(\B{\ubar G}, \cC).
  \end{equation}
  When one specializes to $\cC=\Vect_2$, the functor bicategory
  $\Fun(\B{\ubar G}, \cC)$ is also known as $\Rep_2(G)$, the
  bicategory of 2-representations of $G$. Thus, we have an equivalence
  of bicategories $\Vect_2^{G} \cong \Rep_2(G)$. This result
  generalizes the 1-categorical statement that the homotopy fixed
  point 1-category of the trivial $G$-action on $\Vect$ is equivalent
  to $\Rep(G)$, cf. \cite[Example 4.15.2]{egno-tensor-book}.
\end{ex}

\section{Homotopy fixed points of the trivial SO(2)-action}
\label{sec:results}
We are now in the position to state and prove the main result of the
present paper. Applying the description of homotopy fixed points in
Remark \ref{rem:unpacking-fixed-point} to the trivial action of the
topological group $SO(2)$ on an arbitrary bicategory yields Theorem
\ref{thm:so2-fixed-point-bicat}. Specifying the bicategory in question
to be the core of the fully-dualizable objects of the
Morita-bicategory $\Alg_2$ then shows in corollary
\ref{cor:frob-fixed-points} that homotopy fixed points of the trivial
$SO(2)$-action on $\core{\Alg_2^\fd}$ are given by symmetric,
semi-simple Frobenius algebras.
\begin{theorem}
  \label{thm:so2-fixed-point-bicat}
  Let $\cC$ be a bicategory, and let $\rho: \Pi_2(SO(2)) \to
  \Aut(\cC)$ be the trivial $SO(2)$-action on $\cC$. Then, the
  bicategory of homotopy fixed points $\cC^{SO(2)}$ is equivalent to
  the bicategory where
  \begin{itemize}
  \item objects are given by pairs $(c, \lambda)$ where $c$ is an
    object of $\cC$, and $\lambda: \id_c \to \id_c$ is a
    2-isomorphism,
  \item 1-morphisms $(c, \lambda) \to (c', \lambda')$ are given by
    1-morphisms $f: c \to c'$ in $\cC$, so that the diagram of
    2-morphisms
    \begin{equation}
      \label{eq:1-morphism-fixed-point-modification-condition-1}
      \vcenter{\hbox{\includegraphics{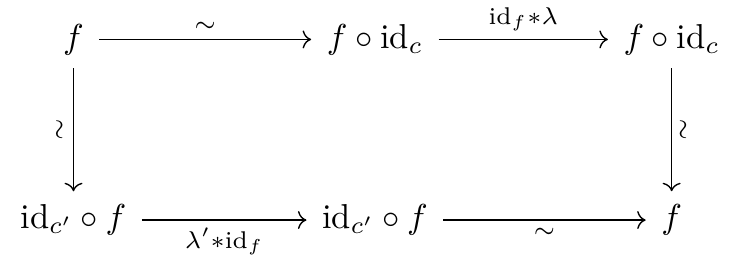}}}
    \end{equation}
    commutes, where $*$ denotes horizontal composition of
    2-morphisms. The unlabeled arrows are induced by the canonical
    coherence isomorphisms of $\cC$.
  \item 2-morphisms of $\cC^G$ are given by 2-morphisms $\alpha: f \to
    f'$ in $\cC$.
  \end{itemize}
\end{theorem}
\begin{proof}
  First, notice that we do not require any conditions on the
  2-morphisms of $\cC^{SO(2)}$. This is due to the fact that the
  action is trivial, and that $\pi_2(SO(2))=0$. Hence, all naturality
  conditions with respect to 2-morphisms in $\Pi_2(SO(2))$ are
  automatically fulfilled.

  To start, we observe that the fundamental 2-groupoid $\Pi_2(SO(2))$
  is equivalent to the bicategory consisting of only one object, $\Z$
  worth of morphisms, and only identity 2-morphisms which we denote by
  $\B{\ubar \Z}$. Thus, it suffices to consider the homotopy fixed
  point bicategory of the trivial action $\B{\ubar \Z} \to
  \Aut(\cC)$. In this case, the definition of a homotopy fixed point
  as in \ref{def:fixed-point} reduces to
  \begin{itemize}
  \item An object $c$ of $\cC$,
  \item A 1-equivalence $\Theta:= \Theta_* : c \to c$,
  \item For every $n \in \Z$, an invertible 2-morphism $\Theta_n :
    \id_c \circ \Theta \to \Theta \circ \id_c$. Since $\Theta$ is a
    pseudo-natural transformation, it is compatible with respect to
    composition of 1-morphisms in $\B \ubar{\Z}$. Therefore,
    $\Theta_{n+m}$ is fully determined by $\Theta_n$ and $\Theta_m$,
    cf. \cite[Figure A.1]{schommerpries-classification} for the
    relevant commuting diagram. Thus, it suffices to specify
    $\Theta_1$.

    By using the canonical coherence isomorphisms of $\cC$, we see
    that instead of giving $\Theta_1$, we can equivalently specify an
    invertible 2-morphism
    \begin{equation}
      \tilde \lambda: \Theta \to \Theta.
    \end{equation}
    which will be used below.
  \item A 2-isomorphism
    \begin{equation}
      \id_c \circ \Theta \circ \Theta \to \Theta
    \end{equation}
    which is equivalent to giving a 2-isomorphism
    \begin{equation}
      \Pi: \Theta \circ \Theta \to \Theta .
    \end{equation}
  \item A 2-isomorphism
    \begin{equation}
\label{eq:m}
      M: \Theta \to {\id_c}.
    \end{equation}
  \end{itemize}
  Note that equivalently to the 2-isomorphism $\tilde \lambda$, one
  can specify an invertible 2-iso\-morphism
  \begin{equation}
    \lambda : \id_c \to \id_c
  \end{equation}
  where
  \begin{equation}
    \label{eq:lambda}
    \lambda:= M \circ  \tilde \lambda \circ M^{-1}.
  \end{equation}
with $M$ as in equation \eqref{eq:m}.
  This data has to satisfy the following three equations: Equation
  \eqref{eq:fixed-point-condition1} says that we must have
  \begin{equation}
    \label{eq:fixed-point-condition1-1}
    \Pi \circ (\id_\Theta * \Pi)  = \Pi \circ (\Pi * \id_\Theta)
  \end{equation}
  whereas equation \eqref{eq:fixed-point-condition2} demands that
  $\Pi$ equals the composition
  \begin{equation}
    \label{eq:fixed-point-condition2-1}
    \Theta \circ \Theta \xrightarrow{ \id_\Theta * M} \Theta \circ \id_c \cong \Theta
  \end{equation} 
  and finally equation \eqref{eq:fixed-point-condition3} tells us that
  $\Pi$ must also be equal to the composition
  \begin{equation}
    \label{eq:fixed-point-condition3-1}
    \Theta \circ \Theta \xrightarrow{ M *  \id_\Theta }\id_c \circ \Theta \cong \Theta.
  \end{equation} 
  Hence $\Pi$ is fully specified by $M$. An explicit calculation using
  the two equations above then confirms that equation
  \eqref{eq:fixed-point-condition1-1} is automatically
  fulfilled. Indeed, by composing with $\Pi^{-1}$ from the right, it
  suffices to show that $\id_\Theta * \Pi=\Pi* \id_\Theta$. Suppose
  for simplicity that $\cC$ is a strict 2-category. Then,
  \begin{equation}
    \begin{aligned}
      \id_\Theta * \Pi & = \id_\Theta *(M* \id_\Theta) && \qquad \text{by equation \eqref{eq:fixed-point-condition3-1}} \\
      &= (\id_\Theta*M)*\id_{\Theta} && \\
      &=\Pi* \id_\Theta && \qquad \text{by equation
        \eqref{eq:fixed-point-condition2-1} }
    \end{aligned}
  \end{equation}
  Adding appropriate associators shows that this is true in a general
  bicategory.

  If $(c,\Theta, \lambda, \Pi,M)$ and $(c',\Theta', \lambda',\Pi',M')$
  are two homotopy fixed points, the definition of a 1-morphism of
  homotopy fixed points reduces to
  \begin{itemize}
  \item A 1-morphism $f: c \to c'$ in $\cC$,
  \item A 2-isomorphism $ m: f \circ \Theta \to \Theta' \circ f$ in
    $\cC$
  \end{itemize}
  satisfying two equations.  The condition due to equation
  \eqref{eq:1-morphism-fixed-point-condition2} demands that the
  following isomorphism
  \begin{align}
    f \circ \Theta \xrightarrow{\id_f * M} f \circ \id_c \cong f
    \intertext{is equal to the isomorphism} f \circ \Theta
    \xrightarrow{m} \Theta' \circ f \xrightarrow{M' * \id_f} \id_{c'}
    \circ f \cong f
  \end{align}
  and thus is equivalent to the equation
  \begin{equation}
    \label{eq:1-morphism-fixed-point-condition2-1}
    m=\left( f \circ \Theta \xrightarrow{\id_f * M} f \circ \id_c \cong f \cong \id_{c'} \circ f \xrightarrow{{M'}^{-1}* \id_{f}} \Theta' \circ f \right) .
  \end{equation}
  Thus, $m$ is fully determined by $M$ and $M'$.  The condition due to
  equation \eqref{eq:1-morphism-fixed-point-condition1} reads
  \begin{equation}
    \label{eq:1-morphism-fixed-point-condition1-1}
    m \circ (\id_f * \Pi) = (\Pi' * \id_f) \circ (\id_{\Theta'} * m) \circ (m * \id_\Theta)
  \end{equation}
  and is automatically satisfied, as an explicit calculation
  confirms. Indeed, if $\cC$ is a strict 2-category we have that
  \begin{align*}
    & (\Pi' * \id_f) \circ (\id_{\Theta'} * m) \circ (m * \id_\Theta)  \\
    & \qquad = (\Pi' * \id_f) \circ \left[ \id_{\Theta'} *({M'}^{-1} * \id_f \circ \id_f* M) \right] \circ \left[ ({M'}^{-1} * \id_f \circ \id_f *M) * \id_\Theta \right] \\
    & \qquad
    \begin{aligned}
      =(\Pi' * \id_f) &\circ (\id_{\Theta'} * {M'}^{-1} * \id_f) \circ (\id_{\Theta'} * \id_f * M) \\
      & \circ ({M'}^{-1} * \id_f * \id_\Theta) \circ (\id_f*M*
      \id_\Theta)
    \end{aligned}
    \\
    & \qquad = (\Pi' *\id_f) \circ ({\Pi'}^{-1}*\id_f) \circ (\id_{\Theta'} * \id_f *M) \circ ({M'}^{-1}* \id_f * \id_\Theta)  \circ (\id_f * \Pi)  \\
    &\qquad= (\id_{\Theta'} * \id_f *M) \circ ({M'}^{-1}* \id_f * \id_\Theta) \circ (\id_f * \Pi) \\
    &\qquad =(M^{-1}* \id_f) \circ (\id_f *M) \circ (\id_f * \Pi) \\
    &\qquad =m \circ (\id_f * \Pi)
  \end{align*}
  as desired. Here, we have used equation
  \eqref{eq:1-morphism-fixed-point-condition2-1} in the first and last
  line, and equations \eqref{eq:fixed-point-condition2-1} and
  \eqref{eq:fixed-point-condition3-1} in the third line.  Adding
  associators shows this for an arbitrary bicategory.

  The condition that $m$ is a modification as spelled out in equation
  \eqref{eq:1-morphism-fixed-point-modification-condition} demands
  that
  \begin{equation}
    (\tilde \lambda' * \id_f) \circ m=m \circ (\id_f *  \tilde \lambda)
  \end{equation}
  as equality of 2-morphisms between the two 1-morphisms
  \begin{equation}
    f \circ \Theta \to \Theta' \circ f.
  \end{equation}
  Using equation \eqref{eq:1-morphism-fixed-point-condition2-1} and
  replacing $\tilde \lambda$ by $\lambda$ as in equation
  \eqref{eq:lambda}, we see that this requirement is equivalent to the
  commutativity of diagram
  \eqref{eq:1-morphism-fixed-point-modification-condition-1}.
 
  If $(f,m)$ and $(g,n)$ are two 1-morphisms of homotopy fixed points,
  a 2-morphism of homotopy fixed points consists of a 2-morphisms
  $\alpha: f \to g$. The condition coming from equation
  \eqref{eq:2-morphism-fixed-point-condition} then demands that the
  diagram
  \begin{equation}
    \label{eq:2-morphism-fixed-point-condition-1}
    \begin{tikzcd}[column sep=large, row sep=large]
      f \circ \Theta \rar{m} \dar[swap]{\alpha * \id_{\Theta}} &  \Theta' \circ f \dar{\id_{\Theta' } * \alpha} \\
      g \circ \Theta \rar[swap]{n} & \Theta' \circ g
    \end{tikzcd}
  \end{equation}
  commutes. Using the fact that both $m$ and $n$ are uniquely
  specified by $M$ and $M'$, one quickly confirms that the diagram
  commutes automatically.

  Our analysis shows that the forgetful functor $U$ which forgets the
  data $M$, $\Theta$ and $\Pi$ on objects, which forgets the data $m$
  on 1-morphisms, and which is the identity on 2-morphisms is an
  equivalence of bicategories. Indeed, let $(c,\lambda)$ be an object
  in the strictified homotopy fixed point bicategory. Choose
  $\Theta:=\id_c$, $M:=\id_\Theta$ and $\Pi$ as in equation
  \eqref{eq:fixed-point-condition2-1}. Then, $U(c,\Theta, M , \Pi,
  \lambda)=(c,\lambda)$. This shows that the forgetful functor is
  essentially surjective on objects.  Since $m$ is fully determined by
  $M$ and $M'$, it is clear that the forgetful functor is essentially
  surjective on 1-morphisms.  Since
  \eqref{eq:2-morphism-fixed-point-condition-1} commutes
  automatically, the forgetful functor is bijective on 2-morphisms and
  thus an equivalence of bicategories.
\end{proof}
In the following, we specialise Theorem
\ref{thm:so2-fixed-point-bicat} to the case of symmetric Frobenius
algebras and Calabi-Yau categories.
\subsection{Symmetric Frobenius algebras as homotopy fixed points}
In order to state the next corollary, recall that the fully-dualizable
objects of the Morita bicategory $\Alg_2$ consisting of algebras,
bimodules and intertwiners are precisely given by the
finite-dimensional, semi-simple algebras
\cite{schommerpries-classification}. Furthermore, recall that the core
$ \core{\cC}$ of a bicategory $\cC$ consists of all objects of $\cC$,
the 1-morphisms are given by 1-equivalences of $\cC$, and the 2-morphisms are
restricted to be isomorphisms.
\begin{cor}
  \label{cor:frob-fixed-points}
  Suppose $\cC=\core{\Alg_2^{\fd}}$, and consider the trivial
  $SO(2)$-action on $\cC$. Then $\cC^{SO(2)}$ is equivalent to the
  bicategory of finite-dimensional, semi-simple symmetric Frobenius
  algebras $\Frob$, as defined in definition
  \ref{def:frob-bicategory}. This implies a bijection of
  isomorphism-classes of symmetric, semi-simple Frobenius algebras and
  homotopy fixed points of the trivial $SO(2)$-action on
  $\core{\Alg_2^\fd}$.
\end{cor}
\begin{proof}
  Indeed, by Theorem \ref{thm:so2-fixed-point-bicat}, an object of
  $\cC^{SO(2)}$ is given by a finite-dimensional semisimple algebra
  $A$, together with an isomorphism of Morita contexts $\id_A \to
  \id_A$. By definition, a morphism of Morita contexts consists of two
  intertwiners of $(A,A)$-bimodules $\lambda_1, \lambda_2: A \to
  A$. The diagrams in definition \ref{def:morphism-of-morita-context}
  then require that $\lambda_1= \lambda_2^{-1}$. Thus, $\lambda_2$ is
  fully determined by $\lambda_1$. Let $\lambda:= \lambda_1$. Since
  $\lambda$ is an automorphism of $(A,A)$-bimodules, it is fully
  determined by $\lambda(1_A) \in Z(A)$. This gives $A$, by Lemma
  \ref{lem:symmetric-frobenius-form-unique-up-to-multiple}, the
  structure of a symmetric Frobenius algebra.

  We analyze the 1-morphisms of $\cC^{SO(2)}$ in a similar way: if
  $(A,\lambda)$ and $(A',\lambda')$ are finite-dimensional semi-simple
  symmetric Frobenius algebras, a 1-morphism in $\cC^{SO(2)}$ consists
  of a Morita context $\cM :A \to A'$ so that
  \eqref{eq:1-morphism-fixed-point-modification-condition-1} commutes.

  Suppose that $\cM=({_{A'}M_A}, {_{A}N_{A'}}, \eps, \eta)$ is a
  Morita context, and let $a:=\lambda(1_A)$ and
  $a':={\lambda'}(1_{A'})$.
  Then, the condition that
  \eqref{eq:1-morphism-fixed-point-modification-condition-1} commutes
  demands that
  \begin{equation}
    \begin{aligned}
      m.a & =a'.m \\
      a^{-1}.n &=n.{ a'}^{-1}
    \end{aligned}
  \end{equation}
  for every $m \in M$ and every $n \in N$.  By proposition
  \ref{prop:morita-context-compatible-equivalence} this condition is
  equivalent to the fact that the Morita context is compatible with
  the Frobenius forms as in definition \ref{def:morita-compatible}.

  It follows that the 2-morphisms of $\cC^{SO(2)}$ and $\Frob$ are
  equal to each other, proving the result.
\end{proof}
\begin{remark}
\label{rem:davidovich}
In \cite[Proposition 3.3.2]{davi11}, the bigroupoid $\Frob$ of
corollary \ref{cor:frob-fixed-points} is shown to be equivalent to the
bicategory of 2-functors $\Fun(B^2\Z ,\core{\Alg_2^\fd})$. Assuming a
homotopy hypothesis for bigroupoids, as well as an equivariant homotopy
hypothesis in a bicategorical framework, this bicategory of functors
should agree with the bicategory of homotopy fixed points of the
trivial $SO(2)$-action on $\core{\Alg_2^\fd}$ in corollary
\ref{cor:frob-fixed-points}. Concretely, one might envision the
following strategy for an alternative proof of corollary
\ref{cor:frob-fixed-points}, which should roughly go as follows:
\begin{enumerate}
\item By \cite[Proposition 3.3.2]{davi11}, there is an equivalence of
  bigroupoids $\Frob \cong \Fun(B^2 \Z, \core{\Alg_2^\fd}).$
\item Then, use the homotopy hypothesis for bigroupoids. By this, we
  mean that the fundamental 2-groupoid should induce an equivalence of
  tricategories
  \begin{equation}
    \label{eq:homohyp}
    \Pi_2 :\Top_{\leq 2} \to \BiGrp.
  \end{equation}
  Here, the right hand-side is the tricategory of bigroupoids, whereas
  the left hand side is a suitable tricategory of 2-types. Such an
  equivalence of tricategories induces an equivalence of bicategories
  \begin{equation}
    \Fun(B^2 \Z, \core{\Alg_2^\fd}) \cong \Pi_2(\Hom(BSO(2), X)),
  \end{equation}
  where $X$ is a 2-type representing the bigroupoid
  $\core{\Alg_2^\fd}$.
\item Now, consider the trivial homotopy $SO(2)$-action on the 2-type
  $X$. Using the fact that we work with the trivial $SO(2)$-action, we
  obtain a homotopy equivalence $ \Hom(BSO(2), X) \cong
  X^{hSO(2)}$, cf.  \cite[Page 50]{davi11}.
\item In order to identify the 2-type $X^{hSO(2)}$ with our definition
  of homotopy fixed points, we additionally need an equivariant homotopy
  hypothesis: namely, we need to use that a homotopy action of a
  topological group $G$ on a 2-type $Y$ is equivalent to a $G$-action
  on the bicategory $\Pi_2(Y) $ as in definition
  \ref{def:action-on-bicat} of the present paper. Furthermore, we also
  need to assume that the fundamental 2-groupoid is $G$-equivariant,
  namely that there is an equivalence of bicategories $\Pi_2(Y^{hG})
  \cong \Pi_2(Y)^G$.  Using this equivariant homotopy hypothesis for
  the trivial $SO(2)$-action on the 2-type $X$ then should give an
  equivalence of bicategories
  \begin{equation}
\label{eq:equivariant-homohyp}
    \Pi_2(X^{hSO(2)}) \cong \Pi_2(X)^{SO(2)} \cong (\core{\Alg_2^\fd})^{SO(2)}.
  \end{equation}
\end{enumerate}
Combining all four steps gives an equivalence of bicategories
between the bigroupoid of Frobenius algebras and homotopy fixed points:
 \begin{equation*}
  \Frob \underset{(1)}{\cong} \Fun(B^2 \Z,  \core{\Alg_2^\fd})
  \underset{(2)}{\cong} \Pi_2(\Hom(BSO(2), X)) \underset{(3)}{\cong}
  \Pi_2(X^{hSO(2)})  \underset{(4)}{ \cong} (\core{\Alg_2^\fd})^{SO(2)}.
  \end{equation*}
  In order to turn this argument into a full proof, we would need to provide
  a proof of the homotopy hypothesis for bigroupoids in equation
  \eqref{eq:homohyp}, as well as a proof for the equivariant homotopy
  hypothesis in equation \eqref{eq:equivariant-homohyp}. While the
  homotopy hypothesis as formulated in equation \eqref{eq:homohyp} is
  widely believed to be true, we are not aware of a proof of this
  statement in the literature.  A step in this direction is
  \cite{MoSv93}, which proves that the homotopy categories of 2-types
  and 2-groupoids are equivalent. We however really need the full
  tricategorical version of this statement as in equation
  \eqref{eq:homohyp}, since we need to identify the (higher) morphisms
  in $\BiGrp$ with (higher) homotopies. Notice that statements of this
  type are rather subtle, see \cite{KV91, Si98}.

  While certainly interesting and conceptually illuminating, a proof
  of the equivariant homotopy hypothesis in a bicategorical language
  in equation \eqref{eq:equivariant-homohyp} is beyond the scope of
  the present paper, which aims to give an \emph{algebraic} description
  of homotopy fixed points on bicategories. Although an equivariant
  homotopy hypothesis for $\infty$-groupoids follows from
  \cite[Theorem 4.2.4.1]{Lu09}, we are not aware of a proof of the
  bicategorical statement in equation \eqref{eq:equivariant-homohyp}.
\end{remark}
Next, we compute homotopy fixed points of the trivial $SO(2)$-action
on $\Vect_2^\fd$ and show that they are given by Calabi-Yau
categories. This result is new and has not yet appeared in the
literature.
\subsection{Calabi-Yau categories as homotopy fixed points}
We now apply Theorem \ref{thm:so2-fixed-point-bicat} to Calabi-Yau
categories, as considered in \cite{morre-segal}. Let $\Vect_2$ be the
bicategory consisting of linear, abelian categories, linear functors,
and natural transformations. 

Recall that a $\K$-linear, abelian category $\cC$ is called finite, if
is has finite-dimensional $\Hom$-spaces, every object has got finite
length, the category $\cC$ has got enough projectives, and there are
only finitely many isomorphism classes of simple objects.

The fully-dualizable objects of $\Vect_2$
are then precisely the finite, semi-simple linear categories,
cf. \cite[Appendix A]{bdsv15-2}.
For convenience, we recall the definition of a finite Calabi-Yau category.

\begin{newdef}
  \label{def:calabi-yau-cat}
  Let $\K$ be an algebraically closed field.  A Calabi-Yau category
  $(\cC, \tr^\cC)$ is a $\K$-linear, finite, semi-simple category
  $\cC$, together with a family of $\K$-linear maps
  \begin{equation}
    \tr_c^\cC: \End_\cC(c) \to  \K
  \end{equation}
  for each object $c$ of $\cC$, so that:
  \begin{enumerate}
  \item for each $f \in \Hom_\cC(c,d)$ and for each $g \in
    \Hom_\cC(d,c)$, we have that
    \begin{equation}
      \tr^\cC_c(g \circ f) = \tr^\cC_d (f \circ g), 
    \end{equation}
  \item for each $f \in \End_\cC(x)$ and each $g \in \End_\cC(d)$, we
    have that
    \begin{equation}
         \tr_{c \oplus d }^\cC (f \oplus g) = \tr_c^\cC(f) + \tr_d^\cC(g),
       \end{equation}
\item for all objects $c$ of $\cC$, the induced pairing
  \begin{align}
    \begin{aligned}
      \langle - \, , - \rangle_\cC: \Hom_\cC(c,d) \ot_\K \Hom_\cC(d,c) &\to \K      \\
      f \ot g& \mapsto \tr_c^\cC(g \circ f)
    \end{aligned}
  \end{align}
is a non-degenerate pairing of $\K$-vector spaces.
 \end{enumerate}
 We will call the collection of morphisms $\tr_c^\cC$ a \emph{trace}
 on $\cC$.

 An equivalent way of defining a Calabi-Yau structure on a linear
 category $\cC$ is by specifying a natural isomorphism
  \begin{equation}
    \Hom_\cC(c,d) \to \Hom_\cC(d,c)^*,
  \end{equation}
cf. \cite[Proposition 4.1]{schau12}. 
\end{newdef}
\begin{newdef}
\label{def:calabi-yau-functor}
Let $(\cC, \tr^\cC)$ and $(\cD, \tr^\cD)$ be two Calabi-Yau
categories. A linear functor $F: \cC \to \cD$ is called a Calabi-Yau
functor, if
    \begin{equation}
       \tr_c^\cC(f)=\tr_{F(c)}^\cD(F(f))
  \end{equation}
  for each $f\in \End_\cC(c)$ and for each $c \in \Ob(\cC)$.
  Equivalently, one may require that
  \begin{equation}
        \langle F f, Fg \rangle_\cD = \langle f, g \rangle_\cC 
  \end{equation}
  for every pair of morphisms $f:c \to d$ and $g: d \to c$ in $\cC$.

  If $F$, $G: \cC \to \cD$ are two Calabi-Yau functors between
  Calabi-Yau categories, a Calabi-Yau natural transformation is just
  an ordinary natural transformation.

  This allows us to define the symmetric monoidal bicategory $\CY$
  consisting of Calabi-Yau categories, Calabi-Yau functors and natural
  transformations. The monoidal structure is given by the Deligne
  tensor product of abelian categories.
\end{newdef}

\begin{cor}
\label{cor:cy-fixed-points}
Suppose $\cC=\core{\Vect_2^{\fd}}$, and consider the trivial
$SO(2)$-action on $\cC$. Then $\cC^{SO(2)}$ is equivalent to the
bicategory of Calabi-Yau categories.
\end{cor}
\begin{proof}
  Indeed, by Theorem \ref{thm:so2-fixed-point-bicat} a homotopy fixed
  point consists of a category $\cC$, together with a natural
  transformation $\lambda:\id_\cC \to \id_\cC$. Let $X_1$, \ldots,
  $X_n$ be the simple objects of $\cC$. Then, the natural
  transformation $\lambda : \id_\cC \to \id_\cC$ is fully determined
  by giving an endomorphism $\lambda_X : X \to X$ for every simple
  object $X$. Since $\lambda$ is an invertible natural transformation,
  the $\lambda_X$ must be central invertible elements in
  $\End_\cC(X)$. Since we work over an algebraically closed field,
  Schur's Lemma shows that $\End_\cC(X) \cong \K$ as vector
  spaces. Hence, the structure of a natural transformation of the
  identity functor of $\cC$ boils down to choosing a non-zero scalar
  for each simple object of $\cC$. This structure is equivalent to
  giving $\cC$ the structure of a Calabi-Yau category.

  Now note that by equation
  \eqref{eq:1-morphism-fixed-point-modification-condition-1} in
  Theorem \ref{thm:so2-fixed-point-bicat}, 1-morphisms of homotopy
  fixed points consist of equivalences of categories $F: \cC \to \cC'$
  so that $F(\lambda_X) = \lambda'_{F(X)}$ for every object $X$ of
  $\cC$. This is exactly the condition saying that $F$ must a
  Calabi-Yau functor.
 
  Finally, one can see that 2-morphisms of homotopy fixed points are
  given by natural isomorphisms of Calabi-Yau functors.
\end{proof}
\newcommand{\etalchar}[1]{$^{#1}$}
\providecommand{\bysame}{\leavevmode\hbox to3em{\hrulefill}\thinspace}
\providecommand{\ZM}{\relax\ifhmode\unskip\space\fi Zbl }
\providecommand{\MR}{\relax\ifhmode\unskip\space\fi MR }
\providecommand{\arXiv}[1]{\relax\ifhmode\unskip\space\fi\href{http://arxiv.org/abs/#1}{arXiv:#1}}
\providecommand{\MRhref}[2]{%
  \href{http://www.ams.org/mathscinet-getitem?mr=#1}{#2}
}
\providecommand{\href}[2]{#2}

\end{document}